\documentclass{amsart}
\usepackage{amssymb}
\usepackage{stmaryrd} 
\usepackage{amsmath} 
\usepackage{amscd}
\usepackage{amsbsy}
\usepackage{commath, longtable}
\usepackage{comment, enumerate}
\usepackage{geometry}
\usepackage[matrix,arrow]{xy}
\usepackage{hyperref}
\usepackage{mathrsfs}
\usepackage{color}
\usepackage{mathtools,caption}
\usepackage[table,dvipsnames]{xcolor}
\usepackage{tikz-cd}
\usepackage{longtable}
\usepackage[utf8]{inputenc}
\usepackage[OT2,T1]{fontenc}
\usepackage[normalem]{ulem}
\usepackage{hyperref}
\usepackage[customcolors]{hf-tikz}
\hypersetup{
 colorlinks=true,
 linkcolor=DarkOrchid,
 filecolor=blue,
 citecolor=olive,
 urlcolor=orange,
 pdftitle={Anti-cyclotomic Heuristics},
 }
\usepackage{booktabs}

\DeclareSymbolFont{cyrletters}{OT2}{wncyr}{m}{n}
\DeclareMathSymbol{\Sha}{\mathalpha}{cyrletters}{"58}

\DeclareMathOperator{\Norm}{\mathsf{N}}

\DeclareMathOperator{\Prob}{Prob}
\DeclareMathOperator{\Gal}{Gal}

\DeclareMathOperator{\disc}{Disc}

\DeclareMathOperator{\Cl}{Cl}

\DeclareMathOperator{\rank}{rank}

\DeclareMathOperator{\Aut}{Aut}

\newcommand{\Zp}{\mathbb{Z}_p}

\newcommand{\cO}{\mathcal O}

\newcommand{\fp}{\mathfrak p}

\newcommand{\ac}{{\mathrm{ac}}}

\newcommand{\Z}{{\mathbb{Z}}}
\newcommand{\Q}{{\mathbb{Q}}}

\newcommand{\F}{{\mathbb{F}}}

\newcommand{\cyc}{{\mathrm{cyc}}}

\newtheorem*{theorem*}{Theorem}
\newtheorem*{conj*}{Conjecture}
\newtheorem*{Ques*}{Question}
\newtheorem{theorem}{Theorem}[section]
\newtheorem{lemma}[theorem]{Lemma}
\newtheorem{proposition}[theorem]{Proposition}
\newtheorem{cor}[theorem]{Corollary}

\definecolor{Green}{rgb}{0.0, 0.5, 0.0}

\theoremstyle{definition}

\theoremstyle{remark}
\newtheorem{remark}[theorem]{Remark}
\newtheorem*{Remark*}{Remark}

\makeatletter
\theoremstyle{plain} %makes the name of the conjecture/theorem bold
\newtheorem*{intr@thm}{\intr@thmname}

\newtheorem*{c@njecture}{\conjn@name}
\newcommand{\myl@bel}[2]{%
  \protected@write \@auxout {}{\string \newlabel {#1}{{#2}{\thepage}{#2}{#1}{}} }%
  \hypertarget{#1}{}
    } %showing an error here
\newenvironment{labelledconj}[3][]%
    {
        \def\conjn@name{#2}
        \begin{c@njecture}[{#1}]\myl@bel{#3}{#2}
    }
    {
        \end{c@njecture}
    }
\makeatother

\begin{document}
\title[Heuristics for  anti-cyclotomic $\Zp$-extensions]{Heuristics for  anti-cyclotomic $\Zp$-extensions}

\author[D.~Kundu]{Debanjana Kundu}
\address[Kundu]{Fields Institute \\ University of Toronto \\
 Toronto ON, M5T 3J1, Canada}
\email{dkundu@math.toronto.edu}

\author[L.~C.~Washington]{Lawrence C.~Washington}
\address[Washington]{Department of Mathematics\\
University of Maryland, College Park, MD 20742 USA}
\email{lcw@umd.edu}

\date{\today}


\begin{abstract}
This paper studies Iwasawa invariants in anti-cyclotomic towers.
We do this by proposing two heuristics supported by computations.
First we propose the \emph{Intersection Heuristics}: these model `how often' the $p$-Hilbert class field of an imaginary quadratic field intersects the anti-cyclotomic tower and to what extent.
Second we propose the \emph{Invariants Heuristics}: these predict that the Iwasawa invariants $\lambda$ and $\mu$ usually vanish for imaginary quadratic fields where $p$ is non-split.
\end{abstract}

\maketitle

\section{Introduction}

Let $K$ be a number field and let $p$ be a prime,  which we shall always take to be odd.
Let $K_{\infty}/K$ be a $\Zp$-extension, which means $\Gal(K_{\infty}/K)\simeq \Zp$, the additive group of $p$-adic integers.
Let $K_n$ be the subfield of degree $p^n$ over $K$, let $h_n$ be its class number, and let $p^{e_n}$ be the exact power of $p$ dividing $h_n$.
A well-known result of K.~Iwasawa says there exist integers $\lambda, \mu\ge 0$ and $\nu$, independent of $n$,  such that 
\[
e_n = \lambda n + \mu p^n + \nu,
\]
for all $n$ sufficiently large.

When $K$ is an imaginary quadratic field, there are two exactly $\Zp$-extensions such that $K_{\infty}/\mathbb Q$ is a Galois extension.
One is the \emph{cyclotomic} $\Zp$-extension $K_{\cyc}/K$: this is contained in the field obtained by adjoining all $p$-power roots of unity to $K$ and is such that $\Gal(K_{\cyc}/\Q)$ is abelian.
The other is the \emph{anti-cyclotomic} $\Zp$-extension $K_{\ac}/K$ and is such that $\Gal(K_{\ac}/\Q)$ is dihedral, with $\Gal(K/\Q)$ acting on $\Gal(K_{\ac}/K)$ by inversion.

The $\mu$-invariant for $K_{\cyc}$ is always 0  (see \cite{FW79}), and the distribution of $\lambda$-invariants when $p$ is fixed and $K$ varies has been studied by several authors, see for example \cite{San91, dummit1991computation, Oza01, Fuj13, Mur19}. 
In \cite{EJV11}, J.~Ellenberg--S.~Jain--A.~Venkatesh studied this question in the special case when $K$ varies over all imaginary quadratic fields and $p$ is non-split in $K$; they proposed a heuristic based on random $p$-adic matrices.
These heuristics have recently been extended to abelian extensions by D.~Delbourgo--H.~Knospe in \cite{DK22}.
The initial goal of our study was to see if similar predictions could be made for $K_{\ac}$.
However, our calculations suggest that $K_{\ac}/K$ where $p$ is non-split in $K$ acts more like the cyclotomic $\Zp$-extensions of totally real number fields, where Greenberg's Conjecture (see \cite{Gre76}) predicts that $\mu=\lambda=0$. 

Let $L(K)$ be the $p$-Hilbert class field of $K$, that is, the maximal unramified abelian $p$-extension of $K$. Then $\Gal(L(K)/K)$ is isomorphic to $A_0$, the Sylow $p$-subgroup of the ideal class group of $K$.
A  phenomenon that plays a prominent role when studying anti-cyclotomic $\Zp$-extensions is that $L(K)\cap K_{\ac}$ can be larger than $K$.
The question is how large can this intersection be?
Based on computational evidence, we propose the following:
\begin{conj*}[Intersection Heuristics]
Fix a finite abelian $p$-group $G$.
Let $K$ vary over imaginary quadratic fields such that the $p$-part of its ideal class group is isomorphic to $G$.
The probability that $L(K)\cap K_{\ac}=K_n$ is given by
\[
\frac{\#(\text{elements of $G$ of order exactly }p^n)}{\#G}.
\]
This number is also
\[
\frac{\#(\text{homomorphisms }G\longrightarrow{\Q/\Z} \text{ with image of size }p^n)}{\# G}.
\]
\end{conj*}
In other words, the $p$-class group $G$ randomly %\DK{surjectively} 
maps to $\Gal(K_n/K)$ for large $n$ and the frequency of possible images is predicted to model the distribution of possible intersections. 

The Cohen-Lenstra heuristics (introduced in \cite{CL}) predict that the $p$-part of the ideal class group of $K$ is most often cyclic, and our heuristic says that for cyclic $G$, it is most common that
$L(K)$ is contained in $K_{\ac}$.
A result of S.~Fujii (see Theorem~\ref{Fujii's thm}) says that $\mu=\lambda=0$ in this case.
Moreover, our calculations suggest that $\mu=0$ is always true, and that
$\lambda=0$ when $A_0$ is cyclic and $L(K)\cap K_{\ac}=K$.
Putting these together yields the following:
\begin{conj*}[Invariants Heuristics]
Among the imaginary quadratic fields $K$ in which $p$ does not split, we always have $\mu=0$, and the proportion for which $\lambda=0$ is at least
\[
\left(1-\frac{1}{p}\right) + \frac{1}{p-1}\prod_{j\geq1}\left(1-p^{-j}\right).
\]
\end{conj*}
For $p=3$, this predicts that $\lambda=0$ at least $94\%$ of the time, and for $p=5$ at least $99\%$ of the time.
These are heuristic lower bounds, but we find it hard to imagine that there is such a small number of actual exceptions.
We show that there is (theoretical and/or experimental) evidence that $\lambda=\mu=0$ in the two extreme cases, i.e., when the $p$-part of the class group of $K$ is contained in the anti-cyclotomic tower or completely disjoint from it.
The cases where we could not calculate the Iwasawa invariants from our experiments are the ones that `interpolate' between these two extreme cases - it seems reasonable to expect these to act similarly.
%I MADE A FEW SMALL EDITS TO THIS.
Therefore we ask the following:
\begin{Ques*}
Is $\lambda=\mu=0$ always true for the anti-cyclotomic $\Zp$-extension of an imaginary quadratic field $K$ in which $p$ does not split?
\end{Ques*}

We restrict ourselves to the case that $p$ does \emph{not} split in $K$ because in the complementary case it is known that $\lambda>0$.
This follows from a result of \cite{Oza01} (see for example \cite[p.~283]{Fuj13}) and can be explained briefly as follows: looking at the inertia groups of the primes above $p$ for the $\Zp^2$-extension $\widetilde{K}/K$, one sees that the $\Zp$-extension $\widetilde{K}/K_{\ac}$ is unramified.
Therefore, the unramified Iwasawa module has positive $\Zp$ rank, so $\lambda>0$.
%Therefore, in this case we have a short exact sequence
%\[
%0 \longrightarrow (X_{\widetilde{K}})_{\Gal(\widetilde{K}/K_{\ac})} \longrightarrow X_{K_{\ac}} \longrightarrow \Zp \longrightarrow 0,
%\]
%where $X_{K_\ac}$ is the Iwasawa module isomorphic to $\Gal(L(K_{\ac})/K_\ac)$; and analogously for $X_{\widetilde{K}}$.
%Thus,
%\[
%\lambda(K_{\ac}/K) = \rank_{\Zp} X_{K_{\ac}} = \rank_{\Zp} (X_{\widetilde{K}})_{\Gal(\widetilde{K}/K_{\ac})} + 1 > 0.
%\]

%This is also explained in \cite[p.~266]{Gre76}.

Another situation where there is possibly an analogue of Greenberg's conjecture was considered in \cite{FK02}.
Suppose $p$ splits in $K$ as $\mathfrak p\overline{\mathfrak p}$.
Let $K_{\infty}/K$ be the $\Zp$-extension in which only $\mathfrak p$ ramifies and assume it is totally ramified.
It is known that $\mu=0$ (see \cite{Gil87, Sch87}), and the computations in \cite{FK02} suggest that also $\lambda=0$.

For cyclotomic $\Zp$-extensions of imaginary quadratic fields, R. Gold was able to evaluate the $\lambda$-invariant from the class numbers of the first few levels, see \cite{Gold}.
We have followed a similar strategy in our situation but it is more challenging.
In the cyclotomic case, Gold used the fact that the inverse limit of the $p$-class groups has no finite submodules.
This is no longer the case for anti-cyclotomic $\Zp$-extensions, as can be seen from the several cases where $\lambda=\mu=0$ but $\nu>0$.
Consequently, our methods cannot recognize situations where $\lambda>0$ or $\mu>0$, if they exist.

In addition to proposing the above heuristics, we prove the following result (see Theorem~\ref{T: no cyclic A_1}) which asserts that the Sylow $p$-subgroup of the ideal class group of the $n$-th layer of the anti-cyclotomic extension, denoted by $A_n$, is non-cyclic if $A_0$ is non-trivial and disjoint from $K_{\ac}$.
\begin{theorem*}
Let $K$ be an imaginary quadratic field and let $p$ be an odd prime that does not split in $K$.
Suppose that $A_0$ is non-trivial.
If the $p$-Hilbert class field of $K$ is disjoint from $K_{\ac}$, then $A_n$ is non-cyclic for all $n\geq 1$.
\end{theorem*}

\emph{Organization:}
Including this Introduction, the article has seven sections.
In Section~\ref{S: Preliminaries}, we record facts from classical Iwasawa theory and remind the reader of the Cohen--Lenstra heuristics.
In Section~\ref{S: Data + Calculations disjoint}, we consider the case when the $p$-Hilbert class field of $K$ is disjoint from the anti-cyclotomic tower.
We present data and prove sufficient conditions for the Iwasawa invariants to be trivial.
In Section~\ref{S: partially disjoint}, we study the case when the $p$-Hilbert class field partially intersects the anti-cyclotomic tower.
The computational methods for obtaining the tables in the above sections are explained in Section~\ref{S: Computation Method}.
In Section~\ref{S: non-cyclicity}, we prove that in the case that the $p$-Hilbert class field of $K$ is disjoint from the anti-cyclotomic tower, the $p$-Hilbert class field of each subsequent layer $K_n$ must be non-cyclic.
Finally, in Section~\ref{S: Heuristics} we present the two heuristics and explain the rationale behind them.
The \ref{H: intersection} model `how often'  the $p$-Hilbert class field intersects the anti-cyclotomic tower and `how much' is the intersection.
The \ref{H: AC} aim to predict `how often' the anti-cyclotomic Iwasawa invariants are trivial.

\section{Preliminaries}
\label{S: Preliminaries}

\subsection{Anti-cyclotomic Iwasawa theory}
Let $K$ be an imaginary quadratic field and $p$ be a fixed prime.
There exists a Galois extension $\widetilde{K}/K$ containing all the $\Zp$-extensions of $K$ such that
\[
\Gal\left(\widetilde{K}/K\right)\simeq \Zp^2.
\]
This $2$-dimensional Galois extension contains uncountably many quotients isomorphic to $\Zp$, and each quotient corresponds to some $\Zp$-extension $K_\infty/K$.
There are two $\Zp$-extensions that are special because they are the only ones that are Galois over $\Q$.
One is the often-studied \emph{cyclotomic} $\Zp$-extension which is abelian over $\Q$; the other is the \emph{anti-cyclotomic} $\Zp$-extension, denoted by $K_\ac$, which a pro-dihedral group (and hence non-abelian).

\begin{lemma}
\label{Fujii's lemma}
Keep the same notations as introduced above.
Then, 
\[
L(K)\cap \widetilde{K} \subseteq K_{\ac}.
\]
\end{lemma}

\begin{proof}
For a proof of this assertion, we refer the reader to \cite[Lemma~2.2]{Fuj13}.
\end{proof}

In view of the above lemma, the following situations may arise:
\begin{enumerate}[(A)]
\item $L(K)=K$.
\item $L(K)\neq K$.
\begin{enumerate}[(a)]
\item $L(K)\cap K_{\ac} = L(K)$, i.e., the $p$-Hilbert class field of $K$ is contained in $K_{\ac}$.
\item $L(K)\cap K_{\ac} = K$, i.e., the $p$-Hilbert class field of $K$ is disjoint from $K_{\ac}$.
\item $K\subsetneq L(K)\cap K_{\ac} \subsetneq L(K)$, i.e., the $p$-Hilbert class field of $K$ partially intersects $K_{\ac}$.
\end{enumerate}
\end{enumerate}

The following theorem shows that the Iwasawa invariants in case (a) are well understood.
\begin{theorem}
\label{Fujii's thm}
With notation as introduced above, if $L(K)\subsetneq K_{\ac}$ then
\[
\mu(K_{\ac}/K) = \lambda(K_{\ac}/K) =0.
\]
\end{theorem}

\begin{proof}
This assertion is proved in \cite[Remarks on p.~286]{Fuj13} using Chevalley's formula.
\end{proof}

\subsection{Cohen--Lenstra Heuristics}
\label{CL heuristics}
The Cohen--Lenstra heuristics are based on the observation that structures often occur in nature with frequency inversely proportional to their number of automorphisms.
The idea of the Cohen--Lenstra heuristics is that the $p$-part of the ideal class group of $K$ (with $p\neq 2$) is a finite abelian $p$-group that is distributed in this way.

\begin{labelledconj}{Cohen--Lenstra Heuristics}{CL Heuristics}
Let $D^-(x)$ denote the set of imaginary quadratic fields $K$ such that $\abs{\disc K}\le x$ and let $G(n)$ denote the set of all finite groups of order $n$.
For any `reasonable' function $f$ on finite abelian groups,
\[
\lim_{x\rightarrow{\infty}}\frac{\sum_{K\in D^-(x)}\abs{f\left(\Cl(K)\right)}}{\#\left\{K\in D^-(x)\right\}} = \lim_{n\rightarrow{\infty}} \frac{\sum_{i=1}^n \sum_{G\in G(i)}\frac{f(G)}{\# \Aut(G)}}{\sum_{i=1}^n \sum_{G\in G(i)}\frac{1}{\# \Aut(G)}}.
\]
\end{labelledconj}
The main examples where these heuristics are applied are when $f$ depends on the Sylow $p$-subgroup of $G$.

\section{Data and initial calculations: \texorpdfstring{$p$-Hilbert class field is disjoint from $K_{\ac}$}{}}
\label{S: Data + Calculations disjoint}

In this section, we consider the case where the $p$-Hilbert class field of $K$ is disjoint from $K_{\ac}$.
For computational reasons, we restrict to $p=3$ (when $p=5$, the extension $K_2/\Q$ has degree 50, and it is hard to compute class numbers of fields of such a large degree).
The methods of computation are described in Section~\ref{S: Computation Method}.

\subsection{Data for the case that \texorpdfstring{$A_0$}{} is cyclic}
Calculations presented in the data below were performed using PARI/GP.
In these tables, we record the $p$-part of the ideal class group in the first few layers of the anti-cyclotomic $\Zp$-extension of $K=\mathbb Q(\sqrt{-d})$, where we take $-d$ to be the  discriminant of the quadratic field.
Here, $A_n$ is the Sylow $p$-subgroup of the ideal class group of $K_n$.
To keep the tables less cumbersome, we write $n$ in place of $\Z/n\Z$.
Since the $p$-Hilbert class field of $K$ is disjoint from $K_{\ac}$, it follows that the unique prime above $p$ is totally ramified in the anti-cyclotomic $\Zp$-extension.
First, in Table~\ref{T: inert} we consider the case when $3$ is inert in the imaginary quadratic field.
In Table~\ref{T: ramify} we present the data when $3$ ramifies in the imaginary quadratic field.

\begin{center}
\renewcommand{\arraystretch}{1.25}
\setlength{\aboverulesep}{0pt}
\setlength{\belowrulesep}{0pt}
\begin{longtable}{ |c|c|c|c||c|c|c|c| }\caption{$3$ is inert}\label{T: inert}\\\toprule
$d$ & $A_0$ & $A_1$ & $A_2$ & $d$ & $A_0$ & $A_1$ & $A_2$ \\ \midrule
331 & 3 & $3\times 3\times 3$ & $3\times 3\times 3\times 3$ & 2491 & 3 & $3\times 3$ & $3\times 3$ \\
643 & 3 & $3\times 3$ & $3\times 3$ & 2740 & 3 & $3\times 3$ & $3\times 3$\\
835 & 3 & $3\times 3$ & $3\times 3$ & 2791 & 3 & $3\times 3$ & $3\times 3$\\
1048 & 3 &  $27\times 9$ & $27\times 27$ & 2824 & 3 & $3\times 3$ & $3\times 3$\\
1192 & 3 & $3\times 3$ & $3\times 3$ & 2923 & 3 & $3\times 3$ & $3\times 3$\\
1327 & 3 & $3\times 3$ & $3\times 3$ & 2344 & 9 & $9\times 3$ & $9\times 9$\\
1588 & 3 & $3\times 3$ & $3\times 3$ & 3643 & 9 & $9\times 3$ & $9\times 9$\\
1843 & 3 & $27\times 9$ & $81\times 27$ & 4363 & 9 & $9\times 3$ & $9\times 9$ \\
1951 & 3 & $3\times 3$ & $3\times 3$ & 4819 & 9 & $9\times 3$ & $9\times 9$\\ 
2227 & 3 & $3\times 3 \times 3$ & $3\times 3\times 3\times 3$ & 5464 & 9 & $9\times 3$ & $9\times 9$\\
2488 & 3 & $3\times 3$ & $3\times 3$ &  6763 & 9 & $9\times 3$ & $9\times 9$ \\
\bottomrule
\end{longtable}
\renewcommand{\arraystretch}{1}
\end{center}

%\vspace{-0.5cm}

\begin{center}
\renewcommand{\arraystretch}{1.25}
\setlength{\aboverulesep}{0pt}
\setlength{\belowrulesep}{0pt}
\begin{longtable}{ |c|c|c|c||c|c|c|c| }\caption{$3$ is ramified}\label{T: ramify}\\\toprule
$d$ & $A_0$ & $A_1$ & $A_2$ & $d$ & $A_0$ & $A_1$ & $A_2$ \\ \midrule
759 & 3 & $3\times 3$ & $3\times 3$ & 3048 & 3 & $9 \times 9$& $9 \times 9$ \\
771 & 3 & $3\times 3$ & $3\times 3$ &  3387 & 3 & $3\times 3$ & $3\times 3$ \\
804 & 3 & $3\times 3$ & $3\times 3$ & 3459 & 3 & $9\times 9$ &  $9\times 9$ \\
1191 & 3 & $3\times 3$ & $3\times 3$ & 3540 & 3 & $3\times 3$ & $3\times 3$ \\
1236 & 3 & $3\times 3$ & $3\times 3$ & 3687 & 3 & $3\times 3$ & $3\times 3$ \\
1272 & 3 & $3\times 3$ & $3\times 3$ & 3783 & 3 & $3\times 3$ & $3\times 3$ \\
1419 & 3 & $9\times 3$ & $27\times 9$ & 3912 & 3  & $3 \times 3$ & $3 \times 3$\\
1515 & 3 & $3\times 3$ & $3\times 3$ & 3999 & 3 & $3 \times 3$ & $3 \times 3$\\
1668 & 3 & $3\times 3$ & $3\times 3$ & 4308 & 3 & $3 \times 3$ & $3 \times 3$\\
2235 & 3 & $3\times 3$ & $3\times 3$ & 4827 & 3 & $3 \times 3$ & $3 \times 3$\\
2283 & 3 & $3\times 3$ & $3\times 3$ & 8331 & 9 & $9\times 3$ & $9\times 9$\\
2355 & 3 & $9\times 9$ & $9\times 9$ & 15243 & 9 & $9\times 3$ & $9\times 9$\\
2856 & 3 & $9\times 9$ & $9\times 9$ & 24207 & 9 & $9\times 3$ & $9\times 9$ \\\bottomrule
\end{longtable}
\renewcommand{\arraystretch}{1}
\end{center}

\begin{theorem}
\label{mu 0 lambda 0 maybe}
In all examples in Tables~\ref{T: inert} and \ref{T: ramify}, we have $\mu(K_\ac/K)=0$.
For all cases except possibly $d=1843$ and $d=1419$, we have $\lambda(K_\ac/K)=0$.
\end{theorem}

\begin{proof}
%The proof below shows that $\mu(K_\ac/K)=0$.
The proofs that $\lambda(K_\ac/K)=0$ for these examples will be given in Sections~\ref{S: Invariants disjoint} and \ref{S: non-cyclic base field data}.
When $K_\ac/K$ is totally ramified,
\cite[Lemma, p.~672]{San91} asserts that 
\begin{equation} \label{eqn: Sands}
\#A_n \ge \#A_0 \cdot p^{\mu(p^n-1)+\text{min}(p^n-1,\lambda)}
\end{equation}
for each $n\ge 0$.
In all of the examples, we observe that
\[
\# A_2 < \# A_0\times 3^{8}.
\]
Therefore, $\mu(K_\ac/K)=0$.
\end{proof}

\subsection{Conditions for triviality of Iwasawa invariants}
\label{S: Invariants disjoint}

In this section, we provide sufficient conditions to prove that the anti-cyclotomic Iwasawa $\mu$ and $\lambda$-invariants are $0$.
Throughout this section we assume that $K$ is an imaginary quadratic field, there exists exactly one ramified prime in the $\Zp$-extension $K_\infty/K$, and it is totally ramified.
In particular, the results we prove hold when the $p$-Hilbert class field of $K$ is disjoint from $K_{\ac}$.

\subsubsection{}
We begin by recording some results that we use several times for proving our statements.
These results are true for \emph{general} $\Zp$-extensions, not just the anti-cyclotomic one, and they apply to arbitrary number fields $K$.

\begin{lemma}
\label{Gerth lemma 4.2}
Suppose that there exists exactly one ramified prime in the $\Zp$-extension $K_\infty/K$; further suppose that it is totally ramified.
Set $G_n = \Gal\left(K_{n+1}/K_n\right)=\langle \tau_n\rangle$.
Then
\[
\# A_{n+1}^{G_n} = \# A_n.
\]
\end{lemma}

\begin{proof}
%This generalizes \eqref{Chevalley} (in Section~\ref{S: non-cyclicity}) and 
This follows from Chevalley's formula; see \cite[Lemma~4.2]{Ger77}.
\end{proof}

\begin{theorem}[{\cite[Theorem~1.3(c)]{Ger77}}]
\label{Gerth thm 1.3 c}
Suppose that there exists exactly one ramified prime in the $\Zp$-extension $K_\infty/K$; further suppose that it is totally ramified.
If every ideal class in $A_0$ becomes trivial in some $A_n$, then
\[
\lambda(K_\infty/K) = \mu(K_\infty/K)=0.
\]
\end{theorem}

The norm map from $A_{n+1}$ to $A_n$ is surjective since we are assuming throughout that $K_{n+1}/K_n$ is totally ramified.
The composition of this with the natural map from $A_n$ to $A_{n+1}$ yields 
\begin{equation}
\begin{split}
\label{eqn: compose}
A_{n+1} &\longrightarrow A_{n+1}\\  
[J] &\mapsto \left( 1+ \tau_n + \ldots + \tau_n^{p-1}\right)[J].
\end{split}
\end{equation}

%\subsubsection{}
%Over the course of the next three results we provide sufficient conditions to ensure that the anti-cyclotomic Iwasawa $\mu$ and %$\lambda$-invariants are trivial when the $p$-Hilbert class field of $K$ is disjoint from the anti-cyclotomic tower.

\begin{proposition}
\label{prop: nakayama}
Suppose that there exists exactly one ramified prime in the $\Zp$-extension $K_\infty/K$; further suppose that it is totally ramified.
Suppose that $A_n\simeq A_{n+1}$ for some $n\geq 0$.
Then $\lambda(K_{\infty}/K)=\mu(K_\infty/K)=0$.
\end{proposition}

\begin{proof}
Write $L=\bigcup_{n\geq} L(K_n)$ and $X=\Gal\left(L/K_\infty\right)$.
Since there is a unique prime above $p$ in $K$ and this prime is totally ramified we know that (see \cite[Proposition~13.22]{Was97})
\[
A_n = X/\left(\left(1+T\right)^{p^n} - 1\right)X.
\]
Since we are assuming that $A_n = A_{n+1}$, therefore
\[
X/\left(\left(1+T\right)^{p^n} - 1\right)X = X/\left(\left(1+T\right)^{p^{n+1}} - 1\right)X.
\]
Nakayama's Lemma now implies that $\left(\left(1+T\right)^{p^n} - 1\right)X=0$; therefore, $X$ must be finite.
\end{proof}

\begin{remark}
This proposition allows us to take care of, for example, the cases $d=2856$ and $3048$ in Table~\ref{T: ramify}.
\end{remark}

\begin{proposition}
\label{prop: elementary}
Suppose that there exists exactly one ramified prime in the $\Zp$-extension $K_\infty/K$; further suppose that it is totally ramified.
Suppose that for some $n$ both $A_n$ and $A_{n+1}$ are elementary $p$-groups satisfying
\[
\rank_p \left(A_{n+1}\right) \leq \rank_p \left(A_{n}\right) + p-2.
\]
Then $\lambda(K_{\infty}/K)=\mu(K_\infty/K)=0$.
\end{proposition}

\begin{proof}
Write $\tau_n$ for a generator of $\Gal(K_{n+1}/K_n)=G_n$.
By Lemma~\ref{Gerth lemma 4.2}, $\# A_{n+1}^{G_n} = \# A_n$.

Consider $A_{n+1}$ as a vector space over $\F_p$ and $\tau_n$ as a linear transformation such that $\tau_n^p =1$.
Note that $1$ is the only eigenvalue of $\tau_n$.

The assumption on ranks implies that the $1$-eigenspace of $\tau_n$ has codimension $\leq p-2$.
If we now look at the Jordan form of $\tau_n$, then no Jordan block has size larger than $p-1$.
Write
\[
\tau_n = I + N \text{ satisfying } N^{p-1}=0.
\]
Hence,
\begin{equation}
\begin{split}
\label{eqn: sum}
1 + \tau_n + \ldots + \tau_n^{p-1} &= I + (I+N) + (I+N)^2 + \ldots + (I+N)^{p-1}\\
& = pI + \sum_{k=1}^{p-1}\sum_{j=1}^{p-2} {{k}\choose{j}}N^j.
\end{split}
\end{equation}
Observe that for a fixed $j\leq p-2$,
\[
\sum_{k=1}^{p-1}{k\choose j} = \sum_{k=j}^{p-1}{k \choose j} = {p\choose {j+1}} \equiv 0 \pmod{p}.
\]
%Since $j\leq p-2$, it follows that numbers less than $j+1$ will never divide $p$.
% https://math.stackexchange.com/questions/726606/combinatorics-proof-of-sum-of-k-choose-m-with-k-from-m-up-to-n-is-equal-to-n#:~:text=to%20the%20top-,Combinatorics%20proof%20of%20%22sum%20of%20(k%20choose%20m)%20with,%2B1%20choose%20m%2B1%22
Putting this all together we rewrite \eqref{eqn: sum} as follows:
\[
1 + \tau_n + \ldots + \tau_n^{p-1}\equiv 0 \pmod{p}.
\]
In particular, the map given in \eqref{eqn: compose} is a multiple of $p$.
By assumption, $A_{n+1}$ is an elementary $p$-group; hence it is annihilated under this map.
But, the norm map is surjective; hence we conclude that $A_n \rightarrow A_{n+1}$ is the $0$-map.
By Theorem~\ref{Gerth thm 1.3 c} the result follows.
\end{proof}

\begin{proposition}
\label{prop: specific}
Suppose that there exists exactly one ramified prime in the $\Zp$-extension $K_\infty/K$; further suppose that it is totally ramified.
Suppose that 
\[
A_0 \simeq \Z/p^2\Z, \quad A_1 = \Z/p^2\Z \times \Z/p\Z, \quad A_2 = \Z/p^2\Z \times \Z/p^2\Z.
\]
Then,
\[
\lambda(K_\infty/K) = \mu(K_\infty/K) =0.
\]
\end{proposition}

\begin{proof}
The elements of $A_1$ can be written as column vectors $\begin{pmatrix}x \\ y \end{pmatrix}$ with $x\in \Z/p^2\Z$ and $y\in \Z/p\Z$.
Let $\tau_1$ be a generator for $\Gal(K_1/K)=G$.
The action of $\tau_1$ on $A_1$ is via the matrix $M = \begin{pmatrix}
a & b\\ c& d
\end{pmatrix}$ such that
\[
\tau_1 \begin{pmatrix}
1 \\ 0
\end{pmatrix} = \begin{pmatrix}
a \\ c
\end{pmatrix} \text{ and }
\tau_1 \begin{pmatrix}
0 \\ 1
\end{pmatrix} = \begin{pmatrix}
b \\ d
\end{pmatrix}.
\]
Since $\begin{pmatrix}
0 \\ 1
\end{pmatrix}$ is an element of order $p$, it forces that $b\equiv 0 \pmod{p}$.

By Lemma~\ref{Gerth lemma 4.2}, we know that
\[
A_1^{G} \simeq \Z/p\Z \times \Z/p\Z \text{ or } A_1^{G} \simeq \Z/p^2\Z.
\]
We claim that in both cases, $1+\tau_1+\tau_1^2+\cdots + \tau_1^{p-2}=p$.
\medskip

\noindent \emph{Case 1 -- $A_1^{G} \simeq \Z/p\Z \times \Z/p\Z$}:
We may choose the fixed elements of $A_1$ to be the column vectors $\begin{pmatrix}
p \\0
\end{pmatrix}$ and $\begin{pmatrix}
0 \\1
\end{pmatrix}$
and this forces
\[
M = \begin{pmatrix}
a & 0\\ c& 1
\end{pmatrix} \text{ such that } a\equiv 1\pmod{p}.
\]
Note that for $0\leq j \leq p-1$,
\[
M^j =  \begin{pmatrix}
a^j & 0\\ \sum_{i=0}^{j-1} a^i c& 1\end{pmatrix} = \begin{pmatrix} a^j & 0 \\ jc & 1\end{pmatrix}
\]
since $a\equiv 1\pmod p$.
Therefore, we calculate that
\[
I + M + \ldots + M^{p-1} = p I.
\]
\noindent \emph{Case 2 -- $A_1^{G} \simeq \Z/p^2\Z$}:
There is a fixed column vector of order $p^2$ and we may choose
\[
M = \begin{pmatrix}
1 & b\\ 0& d
\end{pmatrix}.
\]
Since $M^p = I$, this forces $d=1$.
Note that for $0\leq j \leq p-1$,
\[
M^j =  \begin{pmatrix}
1 & jb\\ 0 & 1
\end{pmatrix}.
\]
Therefore, we obtain that
\[
I + M + \ldots + M^{p-1} = p I.
\]

Continuing with the proof, let $[J_0]$ be an ideal class in $A_0$ and $[J_1]$ be an ideal class in $A_1$ satisfying
\[
\Norm_{K_1/K}([J_1])=[J_0].
\]
Here we have used the fact that $\Norm_{K_1/K}$ is surjective since the extension is totally ramified.
On the other hand, the image of $\Norm_{K_1/K}[J_1]=[J_0]$ in $A_1$ is
\[
\left(I + M + \ldots + M^{p-1} \right)[J_1]=p[J_1].
\]

Suppose that $\langle \tau_2\rangle = \Gal(K_2/K_1)=G_2$ and $A_2\simeq \Z/p^2\Z \times \Z/p^2\Z$.
The action of $\tau_2$ on $A_2$ can be understood via an invertible $2\times 2$ matrix $M_2$ with coefficients in $\Z/p^2\Z$.
Lemma~\ref{Gerth lemma 4.2} implies that there are fixed column vectors of order $p^2$ and order $p$, which we choose to be $\begin{pmatrix}
1 \\ 0
\end{pmatrix}$ and $\begin{pmatrix}
0 \\ p
\end{pmatrix}$, respectively.
Therefore, we choose
\[
M_2 = \begin{pmatrix}
1 & b\\ 0 & d
\end{pmatrix} \text{ with }b\equiv 0 \pmod{p} \text{ and }d\equiv 1\pmod{p}.
\]
Note that once again for $0\leq j \leq p-1$,
\[
M_2^j =  \begin{pmatrix}
1 & \sum_{i=0}^{j-1}bd^i\\ 0 & d^j
\end{pmatrix}
=  \begin{pmatrix}
1 & jb\\ 0 & d^j
\end{pmatrix}.
\]
Therefore, we obtain that
\[
I + M_{2} + \ldots + M_{2}^{p-1} = p I.
\]

Working with the class $[J_1]$ from before, we know there exists a class $[J_2]$ in $A_2$ such that
\[
[J_1] = \Norm_{K_2/K}[J_2].
\]
As before, $[J_1]\mapsto p[J_2]$ in $K_2$.
Therefore,
\[
[J_0] \mapsto p[J_1] \mapsto p^2[J_2]=0.
\]
The result follows from Theorem~\ref{Gerth thm 1.3 c}.
\end{proof}

\begin{remark}\leavevmode
\begin{enumerate}[(i)]
\item Theorem~\ref{mu 0 lambda 0 maybe} and Propositions~\ref{prop: elementary} and \ref{prop: specific} show that $\lambda(K_\ac/K)=\mu(K_\ac/K)=0$ for all imaginary quadratic fields considered in Tables~\ref{T: inert} and \ref{T: ramify} except when $d=1048$, $1419$, or $1843$.
\item The above propositions do not apply to the case $d=1048$.
However, the prime above $7$ generates $A_0$ and this prime becomes principal in $A_2$.
By Theorem~\ref{Gerth thm 1.3 c}, we know that 
\[
\lambda(K_\ac/K) = \mu(K_\ac/K)=0.
\]
\item Now, consider the case $d=1843$.
The prime above $11$ generates $A_0$ but it does not become principal in $A_2$.
So, we can not draw such a conclusion.
The result of Sands mentioned in the proof of Theorem~\ref{mu 0 lambda 0 maybe} implies that $\lambda \le 6$.
It appears possible in this case that $\lambda(K_{\ac}/K)=2$. 
But, it is perhaps unlikely that this is an exceptional case where $\lambda(K_{\ac}/K) > 0$; we have not yet been able to compute the image of $A_0$ in $A_3$.

Similar remarks also hold for the case $d=1419$.
\end{enumerate}
\end{remark}

\subsection{Data for the case that \texorpdfstring{$A_0$}{} is non-cyclic}
\label{S: non-cyclic base field data}
Here are a few cases where $A_0$ is non-cyclic and the $3$-Hilbert class field is disjoint from $K_{\ac}$.

\begin{center}
\renewcommand{\arraystretch}{1.25}
\setlength{\aboverulesep}{0pt}
\setlength{\belowrulesep}{0pt}
\begin{longtable}{ |c|c|c|c| }\caption{$A_0$ is non-cyclic}\label{T: non-cyclic}\\\toprule
$d$ & $A_0$ & $A_1$ & $A_2$ \\ \midrule
50983 & $3\times 3$ & $9 \times 3\times 3\times 3$ & $9\times 9 \times 3\times 3$ \\
63079 & $3\times 3$ & $9 \times 3\times 3\times 3$ & $9\times 9 \times 3\times 3$ \\
141412 & $3\times 3$ & $9 \times 3\times 3\times 3$ & $27\times 9 \times 3\times 3$ \\
\bottomrule
\end{longtable}
\renewcommand{\arraystretch}{1}
\end{center}
The results proven in Section~\ref{S: Invariants disjoint} can not be applied to these examples.
Here, we do not prove general results but only comment on these specific examples.
\begin{enumerate}[(a)]
\item In each of these cases, $\mu(K_\ac/K)=0$ by the same argument as in Theorem~\ref{mu 0 lambda 0 maybe}.
\item When $d=50983$, the $10^{\text{th}}$ powers of the primes above $13$ and $23$ generate $A_0$.
They both become principal in $A_2$.
Theorem~\ref{Gerth thm 1.3 c} implies that 
\[
\mu(K_\ac/K)=\lambda(K_\ac/K)=0.
\]
\item When $d=63079$, the $17^{\text{th}}$ powers of the primes above $5$ and $41$ generate $A_0$.
They both become principal in $A_2$.
Once again, Theorem~\ref{Gerth thm 1.3 c} implies that 
\[
\mu(K_\ac/K)=\lambda(K_\ac/K)=0.
\]
\item When $d=141412$, the $8^{\text{th}}$ powers of the primes above $7$ and $43$ generate $A_0$.
The $8^{\text{th}}$-power of the prime above $43$ becomes principal in $A_2$ but the $8^{\text{th}}$-power of the prime above $7$ has order $3$ in $A_2$.
Therefore, we do not draw any conclusions regarding the $\lambda$-invariant from this.
The result of Sands mentioned in the proof of Theorem~\ref{mu 0 lambda 0 maybe} implies that $\lambda(K_{\ac}/K) \le 5$.
\end{enumerate}

\section{Data and Calculations: \texorpdfstring{$p$-Hilbert class field partially intersects $K_{\ac}$}{}}
\label{S: partially disjoint}

In this section, we first present the table where the $p$-Hilbert class field partially intersects $K_\ac$.

\begin{center}
\renewcommand{\arraystretch}{1.25}
\setlength{\aboverulesep}{0pt}
\setlength{\belowrulesep}{0pt}
\begin{longtable}{ |c|c|c|c||c|c|c|c| }\caption{$p$-Hilbert class field of $K$ intersects with $K_{\ac}$ partially}\label{T: partial}\\\toprule
$d$ & $A_0$ & $A_1$ & $A_2$ & $d$ & $A_0$ & $A_1$ & $A_2$ \\ \midrule
367 & 9 & 3 & $9\times 3$ & 6883 & 9 & 3 & $9\times 3$ \\
1087 & 9 & 3 & $9\times 3$ & 7908 & 9 & 3 & $9 \times 3$\\
1291 & 9 & 3 & $9\times 3$ & 13092 & 9 & 3 & $9 \times 3$\\
3547 & 9 & 3 & $3\times 3\times 3\times 3$ & 36276 & $3\times 3$ & $3\times 3\times 3$ & $9\times 9\times 9\times 3$\\
4012 & 9 & 3 & $9\times 3$ & 49128 & $3\times 3$ & $9\times 3$ & $27\times 9$ \\
6607 & 9 & 3 & $9\times 3$ & 49812 & $3\times 3$ & $9\times 3$ & $27\times 9$ \\
6871 & 9 & 3 & $9\times 3$ & 58920 & $3\times 3$ & $9\times 3$ & $27\times 9$\\
\bottomrule
\end{longtable}
\renewcommand{\arraystretch}{1}
\end{center}

In this next proposition we provide a proof of an observation one is likely to make immediately from the above table.
As before, write $L(K_n)$ to denote the $p$-Hilbert class field of $K_n$.
%the $n$-th layer of the anti-cyclotomic tower.

\begin{proposition}
Suppose that $A_0$ is cyclic and $L(K)\supseteq K_n$ for some $n\geq 0$.
Then $L(K_n) = L(K)$.
\end{proposition}

\begin{proof}
Note that $L(K)\subseteq L(K_n)$.
Suppose that $L(K)\neq L(K_n)$, then
\[
1 \neq \Gal\left( L(K_n)/L(K)\right){\unlhd} \Gal\left( L(K_n)/K\right).
\]
Since these are $p$-groups there exists an index $p$ subgroup $N$ of $\Gal(L(K_n)/L(K))$ that is normal in $\Gal\left( L(K_n)/K\right)$, see \cite[Section 6.1, Theorem 1(3)]{DF}.
Consider the exact sequence
\[
1 \longrightarrow  \Gal\left( L(K_n)/L(K)\right)/N \longrightarrow  \Gal\left( L(K_n)/K\right)/N \longrightarrow  A_0 \longrightarrow  1.
\]
By hypothesis, $A_0$ is cyclic and it acts on $\Gal\left( L(K_n)/L(K)\right)/N$, by lifting and then conjugation.
However, $\Gal\left( L(K_n)/L(K)\right)/N$ has order $p$ which means that there are no automorphisms of order $p$.
The action of $A_0$ must therefore be trivial.
Hence, $\Gal\left( L(K_n)/K\right)/N$ is abelian.
Write $F$ for the number field fixed by $N$.
We have the following tower of number fields:
\[
K \subseteq K_n \subseteq L(K) \subsetneq F \subseteq L(K_n).
\]
This means that $F/K$ is abelian and unramified but this contradicts the assumption that $L(K)$ is the $p$-Hilbert class field.
\end{proof}

\begin{remark}\leavevmode
\begin{enumerate}[(i)]
\item This proposition does not hold when $A_0$ is non-cyclic, as we can see from the data.
\item It seems to be difficult to make any accurate predictions about the anti-cyclotomic Iwasawa invariants from this data.
For $d\ne 36276$, the result of Sands mentioned in the proof of Theorem~\ref{mu 0 lambda 0 maybe}, applied to $K_2/K_1$, says that $\mu(K_{\ac}/K)\le 1$, and if $\mu(K_{\ac}/K)=1$ then $\lambda(K_{\ac}/K)=0$.
For $d=36276$, we find that $\mu(K_{\ac}/K)\le 2$, and if $\mu(K_{\ac}/K)=2$ then $\lambda(K_{\ac}/K)=0$.
\end{enumerate}
\end{remark}

\section{Computational Methods I}
\label{S: Computation Method}
We now explain how the computations were performed to obtain data for the $p$-part of the ideal class groups of the first few layers of the anti-cyclotomic tower in the earlier sections.
Our techniques are a modification of those of \cite{BrHuWa}.

\subsection{}
Let $z$ be an element in the upper half plane and define 
\[
\eta(z) = q^{1/24}\prod_{n\geq 1}(1-q^n) \text{ where } q=e^{2\pi i z}.
\]
Define the \emph{Weber functions}
\[
f(z) = e^{-\pi i/24}\frac{\eta\left(\frac{z+1}{2}\right)}{\eta(z)} \text{ and } f_1(z) = \frac{\eta\left(\frac{z}{2}\right)}{\eta(z)}.
\]

The following result of R.~Schertz \cite[Theorem~1]{Sch02} plays a crucial role in our computations.

\begin{labelledconj}{Schertz's Theorem}{Thm: Schertz}
Let $\alpha$ be an element in the upper half plane satisfying the primitive equation
\[
AX^2 + BX + C =0 \text{ with } 2\nmid A \text{ and } B\equiv 0\pmod{32}
\]
with discriminant $D= B^2 - 4AC = -4m = -t^2 d$.
Then, each of the following numbers generates the ring class of conductor $t$ over the imaginary quadratic field of discriminant $-d$ in the case stated:
\begin{itemize}
\item $\left(\left( \frac{2}{A}\right)\frac{1}{\sqrt{2}}f(\alpha)^2 \right)^3$ if $m\equiv 1\pmod{8}$.
\item $f(\alpha)^3$ if $m\equiv 3\pmod{8}$.
\item $\left( \frac{1}{2}f(\alpha)^4 \right)^3$ if $m\equiv 5\pmod{8}$.
\item $\left(\left( \frac{2}{A}\right)\frac{1}{\sqrt{2}}f(\alpha) \right)^3$ if $m\equiv 7\pmod{8}$.
\item $\left(\left( \frac{2}{A}\right)\frac{1}{\sqrt{2}}f_1(\alpha)^2 \right)^3$ if $m\equiv 2,6\pmod{8}$.
\item $\left(\left( \frac{2}{A}\right)\frac{1}{2\sqrt{2}}f_1(\alpha)^2 \right)^3$ if $m\equiv 4\pmod{8}$.
\end{itemize}
Here, $\left( \frac{2}{A}\right)$ denotes the Legendre symbol.

Let $\alpha_1, \alpha_2, \ldots, \alpha_{h_t}$ be the roots corresponding to the set of all classes of conductor $t$ (satisfying the conditions that the coefficient $A_i$ is odd and $32\mid B_i$), and let $g$ be one of the above expressions.
Then $\{g(\alpha_i)\}_i$ gives a complete set of conjugates of $g(\alpha_1)$ over the imaginary quadratic field of discriminant $-d$.
\end{labelledconj}

Start with a primitive quadratic form 
\[
ax^2 + bxy + cy^2  \text{ with } b^2-4ac = -4m.
\]
If $2\mid a$, we perform the change of variables $y\mapsto (x+y)$ and obtain
\[
(a+b+c)x^2 + (b+2c)xy + cy^2.
\]
Note that $2\mid b$ and the original quadratic form is primitive.
Thus, $c$ and $a+b+c$ are both odd.
Next, suppose we have the quadratic form
\[
a_1x^2 + b_1xy + c_1y^2  \text{ with } 2\nmid a_1.
\]
Choose $m$ satisfying
\[
b_1 + 2a_1 m \equiv 0 \pmod{32}.
\]
Perform the change of variable $x\mapsto x+my$.
This gives the quadratic form
\[
a_1x^2 + (b_1+2a_1 m)xy + (c_1+mb_1+m^2a_1)y^2.
\]
Therefore, each class of quadratic forms is represented as a primitive quadratic form $Ax^2 + Bxy + Cy^2$ with $2\nmid A$ and $32 \mid B$.

Let $\alpha$ be the root in the upper half plane of $Ax^2 + Bx + C =0$.
Compute the value of the appropriate function $g(\alpha)$ from \ref{Thm: Schertz}.
This number generates the ring class field over $K=\Q(\sqrt{-d})$.
The minimal polynomial of $g(\alpha)$ over $\Q$ has roots $g(\alpha_1), \ldots, g(\alpha_{h_t})$.

For our study in the case $p=3$, we need polynomials that generate a layer of the anti-cyclotomic $\mathbb Z_3$-extension of $K$.
Such a layer will be contained in a ring class field of conductor $9^k$ for some $k$.
To use \ref{Thm: Schertz}, multiply $d$ by $4$ (if required) and by a power $9^k$ to get $-4m$.
For $k$ sufficiently large, compute the group of reduced primitive quadratic forms of discriminant $-4m$, call it $G$.
Choose an integer $M$ and let $H$ be the subgroup of $G$ consisting of elements of $G$ whose orders divide $M$.
As we explain below, an appropriate choice of $M$ makes $G/H$ cyclic of order a power of $3$.
The integer $M$ in our calculations was essentially the class number of $K$ times a power of $2$ and a power of $3$.

Define the polynomial
\[
\prod_{aH\in G/H}\left( x- \sum_{h\in H}g\left(\alpha_{ah}\right)\right);
\]
here $g$ is the function in \ref{Thm: Schertz} and $\alpha_{ah}$ is the element of the upper half plane corresponding to the class $ah$.
Observe that $\sum_{h\in H}g\left(\alpha_{ah}\right)$ is the trace to the fixed field of $H$; the roots of the polynomial must lie in this field, which is cyclic over $K$ and of degree a power of $p$.
In all the cases that were considered, these roots generate the field and the polynomial has rational integral coefficients.
We explain in detail below how, from the structure of the ring class group, %(and since we have removed the elements of order dividing the class number of $K$), 
we know that a cyclic $p$-power extension obtained in the manner given above is a layer of the anti-cyclotomic $\mathbb Z_3$-extension.

The calculations were done in floating point (often to 10000- or 20000-
digit accuracy for the larger discriminants) and then rounded to get
a polynomial with integer coefficients (the many-digit accuracy was
needed in order to make the rounding well-defined; GP-Pari keeps track
of this).
The polynomial generating the compositum of $K$ with the field
generated by a root of the polynomial yields a polynomial generating
the layer of the $\mathbb Z_3$-extension over$K$.
As a check, the discriminant of the number field was computed and it always had the correct prime factors to the correct powers.
As the power $9^k$ increases, we get higher layers of the $\mathbb Z_3$-extension.

Here are the details of why we get the anticyclotomic $\mathbb Z_3$ extension by our method.
We assume 3 is not split in $K/\mathbb Q$ and the class field of $K$ intersects the anticyclotomic $\mathbb Z_3$-extension $K_{\infty}$ in $K$.
Moreover, we assume $K\ne \mathbb Q(\sqrt{-3})$ and $K\ne \Q(\sqrt{-1})$.

Let $c$ be a positive integer and let $K_c$ be the ring class field over $K$ of conductor $c$.
Let $G_c=\text{Gal}(K_c/K)$, let $C$ be the ideal class group of $K$, and let $\mathcal O$ be the ring of integers of $K$.
Let $R_c=(\mathcal O/c\mathcal O)^{\times} / (\mathbb Z/c\mathbb Z)^{\times}$.
There is an exact sequence
\[
1 \longrightarrow R_c/{\mathcal O}^{\times} \longrightarrow G_c \longrightarrow C \to  1.
\]
By the Chinese Remainder Theorem, if $c=c_1c_2$ with $\gcd(c_1, c_2)=1$, then
\[
R_{c_1c_2}=R_{c_1}\times R_{c_2}.
\]

Note that the $3$-part of $R_{3^k}/{\mathcal O}^{\times}$ is cyclic of 3-power order or cyclic of 3-power order times $\mathbb Z/3\mathbb Z$.
From the exact sequence, we see that 
\[
\abs{G_c} = \abs{C}\cdot \abs{R_c/{\mathcal O}^{\times}} 
\]
for $c=2$ and $c=2\cdot 3^k$.
It follows that
\[
\abs{G_{2\cdot 3^k}} = \abs{G_2}\cdot \abs{R_{3^k}}.
\]

The prime of $K$ above 3 ramifies in $K_{3^k}/K$ to degree equal to $\abs{R_{3^k}/{\mathcal O}^{\times}}$; therefore, it ramifies at least to this degree in $K_{2\cdot 3^k}/K$.
Let $I_k$ be the inertia group in $G_{2\cdot 3^k}$ for a prime above 3.
Since 3 doesn't ramify in $K_2/K$, we see that
\[
\abs{G_2}\cdot \abs{R_{3^k}}=\abs{G_{2\cdot 3^k}}\ge \abs{G_2}\cdot \abs{I_k}\ge \abs{G_2}\cdot \abs{R_{3^k}/{\mathcal O}^{\times}} = \abs{G_2}\cdot \abs{R_{3^k}}/2,
\]
and therefore $|I_k| = |R_{3^k}|$ or $|I_k| = |R_{3^k}|/2$.

Moreover,
\[
 G_2\times I_k \text{ and } G_2\times R_{3^k}/{\mathcal O}^{\times}
\]
are each either isomorphic to $G_{2\cdot 3^k}$ or to a subgroup of index 2 in it
since these two groups embed in $G_{2\cdot 3^k}$ and the orders are equal up to a factor of 2. 
Since the $3$-part of $R_{3^k}/{\mathcal O}^{\times}$ is either cyclic of 3-power order or $\mathbb Z/3\mathbb Z$ times a cyclic group of 3-power order,
we see that
\[
G_{2\cdot 3^k}\simeq A\times B_k,
\]
where $A$ is independent of $k$ and $B_k$ is cyclic of 3-power order.
This decomposition is not unique, but we can arrange for $B_{k+1}\to B_k$ under the restriction map of 
$G_{2\cdot 3^{k+1}} \to G_{2\cdot 3^k}$ and $I_{k+1}\to I_k$  for each $k$.

Choose $M$ that annihilates $G_{2\cdot 3}$ and let $H_k=\{g\in G_{2\cdot 3^k}\,|\, g^M=1\}$.
Then $A\subseteq H_k$ and
\[
G_{2\cdot 3^k}/H_k\simeq (A\times B_k)/H_k\simeq B_k/(B_k\cap H_k),
\]
where this last group is cyclic of 3-power order.
Under the map $G_{2\cdot 3^{k+1}} \to G_{2\cdot 3^k}$, we have 
\[
B_{k+1}/(B_{k+1}\cap H_{k+1}) \to B_k/(B_k\cap H_k)
\]
for each $k$.
We see that the fixed fields of $B_k/(B_k\cap H)$, as $k$ varies, form the layers of a $\mathbb Z_3$-extension of $K$.
By the Galois action of $\Gal(K/\mathbb Q)$, it must be the anticyclotomic $\mathbb Z_3$-extension of $K$. 

\subsection{}
We now give a concrete example.
Let $d=1048$, $p=3$, and $K=\mathbb Q(\sqrt{-1048})$ which has class number $6$.
Note that $d/4 = 262 \equiv 6\pmod{8}$ so $g(\alpha) = \left(\left( \frac{2}{A}\right)\frac{1}{\sqrt{2}}f_1(\alpha)^2 \right)^3$.

Let $G$ be the group of reduced primitive quadratic forms of discriminant $-1048\times 9^k$ for $k=3,4$ and $H$ be the subgroup of forms of order dividing $24$.
When $k=3$, we have $\#G =216$ and $\#H = 72$.
Therefore, we get a polynomial of degree $3$.
The command \texttt{polredbest} in GP-Pari reduces this polynomial to one with smaller coefficients that generates the same field, namely
\[
x^3 + 42x -40.
\]
This polynomial generates the first layer $K_1/K$.
The polynomial that generates $K_1/\mathbb Q$ is given by
\[
x^6 - 90x^4 -56x^3 +4383x^2 + 4092x + 1046.
\]
When $k=4$, $\# G =648$ and $\# H = 72$.
The degree $9$ polynomial we obtain is
\[
x^9 - 99x^7 - 414x^6 - 2025x^5 + 13086x^4 + 396855x^3 + 830250x^2 + 1026756x - 2776392.
\]
Finally, the degree $18$ polynomial that generates $K_2/\mathbb Q$ is given by
\begin{multline*}
x^{18} - 2970x^{14} - 13080x^{12} + 7331265x^{10} + 259656840x^8 + 3345463512x^6\\
+ 18497139840x^4 + 129166770960x^2 + 1140995344768.
\end{multline*}

\section{\texorpdfstring{Non-cyclicity of $A_n$ when $A_0$ is disjoint from $K_{\ac}$}{}}
\label{S: non-cyclicity}
Recall that in Section~\ref{S: Data + Calculations disjoint} we considered the case that the $p$-Hilbert class field of $K$ is non-trivial and completely disjoint from the anti-cyclotomic tower.
The data in Tables~\ref{T: inert} and \ref{T: ramify} seem to indicate that $A_1$ and $A_2$ are non-cyclic.
In this section, we prove this observation.

\begin{theorem}
\label{T: no cyclic A_1}
Let $K$ be an imaginary quadratic field and let $p$ be an odd prime that does not split in $K$.
Suppose that $A_0$ is non-trivial.
If the $p$-Hilbert class field of $K$ is disjoint from $K_{\ac}$, then $A_n$ is non-cyclic for all $n\geq 1$.
\end{theorem}

The proof will occupy the remainder of this section.
Let $A_n$ denote the Sylow $p$-subgroup of the ideal class group of $K_n$.
Since $K_n/K_0$ is totally ramified, the norm map
\[
A_n \longrightarrow A_0
\]
is surjective.
Therefore, if $A_0$ is non-cyclic, then $A_n$ must be non-cyclic as well.
It suffices to prove the theorem in the case that $A_0$ is cyclic of order $p^u$ where $u\geq 1$.
Once again, the same argument implies that if we show $A_1$ is non-cyclic then $A_n$ must be non-cyclic for all $n\geq 1$.

\begin{remark}
In the case that $p=3$, we are assuming that $A_0\simeq \Z/3^u\Z$.
It follows that we need not worry about the special case when $p=3$ and $K=\Q(\sqrt{-3})$, where $A_0=1$.
Therefore, throughout our discussion, we have the additional property that $\cO_{K}^\times$ has order prime to $p$.
\end{remark}

\subsection{}
\label{general calc}
First, we prove two lemmas that hold for a general $\Zp$-extension of an imaginary quadratic field $K$ in which $p$ is not split and such that $K_1/K$ is ramified.
Write $G:=\Gal\left(K_1/K\right)\simeq \langle\tau\rangle$.
Let $\Norm_{K_1/K}$ denote the norm for $K_1/K$.
By Chevalley's formula (see also Lemma~\ref{Gerth lemma 4.2})
\begin{equation}
\label{Chevalley}
\# A_1^{G} = \frac{\# A_0 \times p}{ p [\cO_{K}^\times : \cO_{K}^\times \cap \Norm_{K_1/K}(K_1^\times)]} = \# A_0.
\end{equation}

\begin{lemma}
With the assumptions introduced above, $\# A_1\neq \# A_0$.
\end{lemma}

\begin{proof}
Suppose that $\# A_1 = \# A_0$.
It follows from \eqref{Chevalley} that $A_1 = A_1^G$.

Denote by $I$ an ideal that represents a class in $A_1$.
Since it is fixed by $\tau$, we have
\[
I^\tau = I\alpha \text{ for some }\alpha\in K_1.
\]
Taking the norm to $K$ on both sides of the equation, we obtain that 
\[
(1) = \left(\Norm_{K_1/K}(\alpha)\right).
\]
Without loss of generality, we may assume that $N(\alpha) = 1$.
By Hilbert's Theorem 90 (for cyclic groups), we know that
%see for example wiki LOL
\[
\alpha = \beta^{1-\tau} \text{ for some }\beta\in K_1.
\]
It follows that
\[
(I\beta)^\tau = I\beta.
\]

Consider the following field diagram.
We write $L(K_i)$ to denote the $p$-Hilbert class field of $K_i$.

\begin{center}
\begin{tikzpicture}[node distance = 1.75cm, auto]
 \node (Q) {$\mathbb{Q}$};
 \node (K) [node distance=.7cm, above of=Q] {$K$};
 \node (K1) [node distance=2cm, above of=K] {$K_1$};
 \node (LK) [node distance=0.8cm, above of=K, left of=K] {$L(K)$};
 \node (C)[node distance=0.8cm, above of=K1, left of=K1] {$L(K)\cdot K_1$};
 \draw[-] (Q) to node {} (K);
 \draw[-] (K) to node {}(K1);
 \draw[-] (LK) to node {} (K);
 \draw[-] (K1) to node {} (C);
 \draw[-] (LK) to node {} (C);
 \draw[bend left=55] (K) to node {\tiny$A_0$} (LK);
 \draw[bend right=25] (K) to node [swap]{\tiny$\Z/p\Z$} (K1);
\end{tikzpicture}
\end{center}
By hypothesis there is a unique prime $\fp\mid p$ in $K$ that ramifies in $K_1$.
Denote by $\mathfrak{P}$ this unique prime above $p$ in $K_1$.
Recall that a prime splits  completely in the $p$-Hilbert class field if and only if its ideal class has order prime to $p$.
For a non-split prime $\fp\mid p$, we know that either $\fp$ or $\fp^2$ is principal in $K$.
Therefore, the ideal class has order $1$ or $2$.
It follows that $\fp$ splits completely in $L(K)/K$.
Note that $\mathfrak{P}\mid p$ splits completely in $L(K)\cdot K_1/K_1$.
We have assumed that $\#A_1 =\# A_0$, so we conclude that $L(K)\cdot K_1 = L(K_1)$ (since $K_1/K_0$ is a ramified extension).
Since we have shown that $\mathfrak{P} \mid p$ splits completely in $L(K_1)$, it has order prime to $p$ in the class group of $K_1$.

Again consider the ideal $I$.
From the above discussion, we conclude that there exists an integer $h'$ that is not divisible by $p$ such that the class of $I^{h'}$ is represented by an ideal from $A_0$.
We choose $h'$ to be the order of the prime-to-$p$ part of the class number of $K_1$.
Therefore, the natural map $A_0 \rightarrow A_1$ is surjective.
But, the norm map $A_1 \rightarrow A_0$ is also surjective because $K_1/K$ is a ramified extension.
This means that the composition
\[
A_0 \longrightarrow A_1 \xrightarrow{\Norm_{K_1/K}} A_0,
\]
which is the $p^{\text{th}}$-power map must be surjective.
But, this is a contradiction.
\end{proof}

\begin{lemma}
With the notation and setting as above, $A_1$ is not cyclic of order $p^k$ with $k\geq u+2$.
\end{lemma}

\begin{proof}\leavevmode
\emph{Claim:} The extension $L(K_1)/K$ is non-abelian.\newline
\emph{Justification:} Suppose that $L(K_1)/K$ is abelian.
Let $\fp$ be the unique prime above $p$ in $K$.
The inertia group of $\fp$ has order $p$; so, the fixed field is an abelian unramified extension of $K$ of degree $p^k$.
But this contradicts the assumption that $\# A_0 =p^u$.

Recall the notation $G:=\Gal\left(K_1/K\right)=\langle \tau\rangle$.
As a consequence of the above claim, we see that $\tau$ gives a non-trivial automorphism of $A_1$ of order $p$.
%In other words, the action of $\tau$ on $A_1$ is non-trivial.

Now suppose that $A_1 \simeq \Z/p^k \Z$.
Note that $\Aut(A_1)$ has a unique subgroup of order p, which is generated by multiplication by $1 + p^{k-1}$.
The fixed group of this subgroup has order $p^{k-1}$ but we have observed before that $\# A_1^{G}=p^u$.
This implies that $k= u+1$.
\end{proof}

\subsection{}
In the remainder of the discussion we will crucially require that we are studying the anti-cyclotomic $\Zp$-extension of $K$.
Denote by $\mathcal{G}$ the Galois group $\Gal\left(L(K_1)/K_0\right)$.
Suppose that $A_1\simeq \Z/p^{u+1}\Z$.
We have the following exact sequence
\[
0 \longrightarrow \Z/p^{u+1}\Z \longrightarrow \mathcal{G} \longrightarrow \langle \tau\rangle \longrightarrow 0.
\]
Since $\tau$ acts on $\Z/p^{u+1}\Z$ as an automorphism of order $p$, we can identify
\[
\langle \tau \rangle \simeq 1+p^u\Z/p^{u+1}\Z.
\]
An element $x\in 1+p^u\Z/p^{u+1}\Z$ acts on $\Z/p^{u+1}\Z$ by multiplication by $x$.
The group $\mathcal{G}$ has the following concrete description
\[
\mathcal{G} \simeq \left\{\begin{pmatrix} a & b\\ 0 & 1\end{pmatrix} \mid a\in 1+p^u\Z/p^{u+1}\Z, \ b\in \Z/p^{u+1}\Z\right\}.
\]
The quotient group $G=\Gal\left(K_1/K\right)=\langle \tau \rangle$ is represented by matrices of the form $\begin{pmatrix} c & 0\\ 0 & 1\end{pmatrix}$.
Note that
\[
\begin{pmatrix} c & 0\\ 0 & 1\end{pmatrix} \begin{pmatrix} a & b\\ 0 & 1\end{pmatrix} \begin{pmatrix} c & 0\\ 0 & 1\end{pmatrix}^{-1} = \begin{pmatrix} c & 0\\ 0 & 1\end{pmatrix} \begin{pmatrix} a & b\\ 0 & 1\end{pmatrix} \begin{pmatrix} 1/c & 0\\ 0 & 1\end{pmatrix} = \begin{pmatrix} a & bc\\ 0 & 1\end{pmatrix}.
\]

\begin{lemma}
\label{lemma to be used next}
With the notation introduced above,
\[
\begin{pmatrix} a & b\\ 0 & 1\end{pmatrix}^p = \begin{pmatrix} 1 & pb\\ 0 & 1\end{pmatrix}
\]
\end{lemma}

\begin{proof}
First, we rewrite
\[
\begin{pmatrix} a & b\\ 0 & 1\end{pmatrix} = \begin{pmatrix} 1 & b\\ 0 & 1\end{pmatrix} \begin{pmatrix} a & 0\\ 0 & 1\end{pmatrix}.
\]
Therefore
\begin{align*}
\begin{pmatrix} a & b\\ 0 & 1\end{pmatrix}^p &= \begin{pmatrix} 1 & b\\ 0 & 1\end{pmatrix} \begin{pmatrix} a & 0\\ 0 & 1\end{pmatrix} \times \begin{pmatrix} 1 & b\\ 0 & 1\end{pmatrix} \begin{pmatrix} a & 0\\ 0 & 1\end{pmatrix} \times \ldots \times \begin{pmatrix} 1 & b\\ 0 & 1\end{pmatrix} \begin{pmatrix} a & 0\\ 0 & 1\end{pmatrix} \\    
& = \begin{pmatrix} 1 & 0\\ 0 & 1\end{pmatrix}\begin{pmatrix} 1 & b\\ 0 & 1\end{pmatrix} \begin{pmatrix} a & 0\\ 0 & 1\end{pmatrix}^{-1} \times \begin{pmatrix} a & 0\\ 0 & 1\end{pmatrix}^2 \begin{pmatrix} 1 & b\\ 0 & 1\end{pmatrix} \begin{pmatrix} a & 0\\ 0 & 1\end{pmatrix}^{-2} \times \ldots \\
&\times \begin{pmatrix} a & 0\\ 0 & 1\end{pmatrix}^{p-1} \begin{pmatrix} 1 & b\\ 0 & 1\end{pmatrix} \begin{pmatrix} a & 0\\ 0 & 1\end{pmatrix}^{-(p-1)} \times \begin{pmatrix} a & 0\\ 0 & 1\end{pmatrix}^{p} \\
& = \prod_{j=0}^{p-1}\begin{pmatrix} a & 0\\ 0 & 1\end{pmatrix}^{j}\begin{pmatrix} 1 & b\\ 0 & 1\end{pmatrix} \begin{pmatrix} a & 0\\ 0 & 1\end{pmatrix}^{-j} \times \begin{pmatrix} a & 0\\ 0 & 1\end{pmatrix}^p\\
& = \prod_{j=0}^{p-1}\begin{pmatrix} 1 & a^jb\\ 0 & 1\end{pmatrix} \times \begin{pmatrix} 1 & 0\\ 0 & 1\end{pmatrix}\\ %p-power gives identity
& = \begin{pmatrix} 1 & \sum_{j=0}^{p-1}a^jb\\ 0 & 1\end{pmatrix}.
\end{align*}
Recall that $a =  1+ kp^u$ for some integer $k$.
Therefore,
\[
\sum_{j=0}^{p-1}a^j  \equiv 1 + \left(1+kp^u\right) + \left(1+2kp^u\right) + \ldots + \left(1+(p-1)kp^u\right) \equiv p \pmod{p^{u+1}}.
\]
This completes the proof of the lemma.
\end{proof}

\begin{lemma}
\label{phi lemma}
There is no automorphism $\phi$ of $\mathcal{G}$ such that for each $a\in 1+p^u\Z/p^{u+1}\Z$, there exists $b_a\in \Z/p^{u+1}\Z$ satisfying
\[
\phi\left( \begin{pmatrix} a & 0\\ 0 & 1\end{pmatrix} \right) = \begin{pmatrix} 1/a & b_a\\ 0 & 1\end{pmatrix}.
\]
\end{lemma}

\begin{proof}
We prove this by contradiction; suppose such an automorphism $\phi$ exists.

Consider the matrix $P = \begin{pmatrix} 1+p^u & 0\\ 0 & 1\end{pmatrix}$.
Then, there exists $b$ such that
\[
\phi\left( P \right) = \begin{pmatrix} 1-p^u & b\\ 0 & 1\end{pmatrix}.
\]
Note that $P$ has order $p$.
Therefore, using Lemma~\ref{lemma to be used next} we see that
\[
\begin{pmatrix} 1 & 0\\ 0 & 1\end{pmatrix} = \begin{pmatrix} 1-p^u & b\\ 0 & 1\end{pmatrix}^p = \begin{pmatrix} 1 & pb\\ 0 & 1\end{pmatrix}.
\]
Consequently, $b\equiv 0\pmod{p^u}$.

Now consider the matrix $M = \begin{pmatrix} 1 & 1\\ 0 & 1\end{pmatrix}$ of order $p^{u+1}$.
Applying the automorphism $\phi$, we obtain
\[
\phi\left( M \right) = \begin{pmatrix} d & e\\ 0 & 1\end{pmatrix}
\]
which must also have order $p^{u+1}$.
Therefore,
\[
\begin{pmatrix} 1 & 0\\ 0 & 1\end{pmatrix} \neq \begin{pmatrix} d & e\\ 0 & 1\end{pmatrix}^{p^u} = \begin{pmatrix} 1 & p^u e\\ 0 & 1\end{pmatrix}.
\]
Consequently, $e\not\equiv 0\pmod{p}$.

Finally, consider
\begin{equation}
\label{consider}
\phi\left(PMP^{-1} \right) = \phi\left( \begin{pmatrix} 1 & 1+p^u\\ 0 & 1\end{pmatrix}\right) = \phi\left(M^{p^u+1}\right).
\end{equation}
We know that
\begin{align*}
\phi\left(PMP^{-1} \right) = \phi(P)\phi(M)\phi(P^{-1}) &= \begin{pmatrix} 1-p^u & b\\ 0 & 1\end{pmatrix} \begin{pmatrix} d & e\\ 0 & 1\end{pmatrix} \begin{pmatrix} 1-p^u & b\\ 0 & 1\end{pmatrix}^{-1}\\
& = \begin{pmatrix} d & e(1-p^u)\\ 0 & 1\end{pmatrix}.
\end{align*}
For the last equality, we have used the fact that $b\equiv 0 \pmod{p^u}$ and $d\equiv 1\pmod{p^u}$.

On the other hand
\begin{align*}
\phi(M^{p^u+1}) = \phi(M)^{p^u}\phi(M) &= \begin{pmatrix} d & e\\ 0 & 1\end{pmatrix}^{p^u} \begin{pmatrix} d & e\\ 0 & 1\end{pmatrix}\\
& = \begin{pmatrix} 1 & p^u e\\ 0 & 1\end{pmatrix} \begin{pmatrix} d & e\\ 0 & 1\end{pmatrix} \quad \text{ by Lemma~\ref{lemma to be used next}}\\
&=\begin{pmatrix} d & e(1+p^u)\\ 0 & 1\end{pmatrix}.
\end{align*}
By \eqref{consider}, it follows that
\[
\begin{pmatrix} d & e(1-p^u)\\ 0 & 1\end{pmatrix} = \begin{pmatrix} d & e(1+p^u)\\ 0 & 1\end{pmatrix}.
\]
Equivalently,
\[
e(1-p^u) \equiv e(1+p^u) \pmod{p^{u+1}}.
\]
This contradicts the fact that $e\not\equiv 0\pmod{p}$.
This completes the proof of the lemma.
\end{proof}

We now prove Theorem~\ref{T: no cyclic A_1}.

\begin{proof}
We know from the calculations done in Section~\ref{general calc} that $A_1$ is not cyclic of order $p^u$ or order $p^k$ where $k\geq {u+2}$.
We are only left to eliminate the case that $A_1$ is cyclic of order $p^{u+1}$.

The concrete description of $A_1$ is as follows:
\[
A_1 = \Gal\left(L(K_1)/K_1\right)  \simeq \left\{\begin{pmatrix} 1 & b\\ 0 & 1\end{pmatrix} \mid b\in \Z/p^{u+1}\Z\right\}.
\]
Complex conjugation gives an element $\phi\in \Aut(\mathcal{G})$ that maps $A_1$ to itself.
Since $K_1$ is contained in $K_{\ac}$, it means that $\phi$ acts by inversion on $\langle \tau \rangle$; i.e.,
\[
\phi\left( \begin{pmatrix} a & 0\\ 0 & 1\end{pmatrix} \right) \equiv \begin{pmatrix} 1/a & 0\\ 0 & 1\end{pmatrix} \pmod{A_1}.
\]
But, such a $\phi$ does not exist by Lemma~\ref{phi lemma}.
This contradicts the assumption that $A_1\simeq \Z/p^{u+1}\Z$.
This completes the proof of the theorem.
\end{proof}

\section{Heuristics}
\label{S: Heuristics}
In this section, we aim to understand, as $K$ varies over all imaginary quadratic fields, how often $L(K)$ is disjoint from $K_{\ac}$, partially intersecting $K_{\ac}$, and completely contained inside $K_{\ac}$.

\subsection{}
Fix an odd prime $p$.
Let $K$ be an imaginary quadratic field and $K_{\ac}$ denote the anti-cyclotomic $\Zp$-extension of $K$.
Let $K_n$ denote the $n$-th layer of the $\Zp$-extension, and $L(K)$ be the $p$-Hilbert class field of $K$.
Write $A_n$ to denote the $p$-part of the ideal class group of $K_n$ for all $n\geq 0$.

\begin{labelledconj}{Intersection Heuristics}{H: intersection}
Fix a finite abelian $p$-group $G$.
Let $K$ vary over imaginary quadratic fields such that the $p$-part of its ideal class group is isomorphic to $G$.
The probability that $L(K)\cap K_{\ac}=K_n$ is given by
\[
\frac{\#(\text{elements of $G$ of order exactly }p^n)}{\#G}.
\]
This number is also
\[
\tag{H1}
\label{H1}
\frac{\#(\text{homomorphisms }G\longrightarrow{\Q/\Z} \text{ with image of size }p^n)}{\# G}.
\]
\label{heuristics for intersection}
\end{labelledconj}
Note that the numerator of \eqref{H1} can also be interpreted as 
\[
\#(\text{surjective homomorphisms }G\longrightarrow{\Z/p^n\Z}).
\]
In particular, we expect that as $K$ varies over imaginary quadratic fields with $A_0\simeq G$, the probability that class field is disjoint from $K_{\ac}$ is $\frac{1}{\# G}$.

We further remark that our calculations suggest that the \ref{H: intersection} is independent of the splitting type.
For example, the heuristics remain the same even if we restrict ourselves to vary over imaginary quadratic fields such that $p$ splits in $K$ and the $p$-part of the ideal class group is isomorphic to $G$.
This agrees with the computational evidence that the Cohen-Lenstra predictions are the same for each splitting type.

\subsubsection{} As $K$ varies over all imaginary quadratic fields, the probability that the class field is disjoint is given by
\begin{equation}
\label{disjoint class field}
\frac{\sum_{G} (\text{frequency of }G)(1/\#G)}{\sum_{G} (\text{frequency of }G)}.
\end{equation}

Set $\Aut(G)$ to denote the group of all automorphisms of $G$.
Define
\[
w(p^k) = \sum_{\# G = p^k} \frac{1}{\# \Aut(G)}.
\]
Here, the sum is over finite abelian $p$-groups of order $p^k$.

It follows from the Cohen--Lenstra heuristics (see for example \cite[Proposition~5.7]{Woo16}) that the denominator of \eqref{disjoint class field} is given by
\[
\sum_{G} (\text{frequency of }G) = \sum_{\alpha=0}^{\infty} w(p^\alpha) = \prod_{j\geq 1} \left(1-p^{-j} \right)^{-1}.
\]
On the other hand, the numerator of \eqref{disjoint class field} is given by
\[
\sum_{G} (\text{frequency of }G)(1/\#G) = \sum_{\alpha=0}^\infty w(p^\alpha)p^{-\alpha} = \prod_{j\geq 2}\left(1-p^{-j} \right)^{-1}.
\]
Therefore, we obtain that as $K$ varies over all imaginary quadratic number fields, the (expected) probability that the ideal class group of $K$ is \emph{completely disjoint} from $K_{\ac}$ is given by 
\begin{equation}
\label{prob of disjoint}
\prod_{j\geq 2}\left(1-p^{-j} \right)^{-1} \times \prod_{j\geq 1}\left(1-p^{-j} \right) = 1 - \frac{1}{p}.
\end{equation}

\begin{remark}
Note that this includes those imaginary quadratic fields for which $p\nmid h_K$ and these fields are a major contribution to the total.
As per the Cohen--Lenstra heuristics, varying over all imaginary quadratic fields the probability that $p\nmid h_k$ is given by 
\[
\prod_{j\geq 1} \left( 1-p^{-j}\right).
\]
\end{remark}

\subsubsection{}
Let us now calculate, as $K$ varies over all imaginary quadratic fields, the probability that the $p$-Hilbert class field is contained in $K_{\ac}$.
For such a $K$, we must have $A_0$ cyclic.

If $G$ is a cyclic $p$-group of order $p^k$ with $k\geq 1$, its automorphism group has order $\phi(p^k)$.
Also,
\[
\#(\text{surjective homomorphisms }G\rightarrow\Z/p^k\Z) = \phi(p^k).
\]
As we vary over all imaginary quadratic fields $K$, the probability that the $p$-class field is (non-trivial) and \emph{completely contained} in $K_{\ac}$ is given by
\begin{multline}
\label{prob Fujii}
\frac{\displaystyle{\sum_{k\geq 1, G \simeq \Z/p^k\Z}} \left(\text{frequency of }G\right)\cdot \frac{\#\left(\text{surjective homomorphisms }G\rightarrow \Z/p^k\Z\right)}{\#G}}{\sum_{G} (\text{frequency of }G)} \\
= \left(\sum_{k\geq 1}\frac{1}{\phi(p^k)}\cdot \frac{\phi(p^k)}{p^k}\right)\prod_{j\geq 1}\left( 1-p^{-j}\right)
= \frac{1}{p-1}\prod_{j\geq1}\left(1-p^{-j}\right).
\end{multline}

\subsection{}
In this section we predict how often $\mu(K_{\ac}/K) = \lambda(K_{\ac}/K)=0$ as $K$ varies over all imaginary quadratic fields such that $p$ does not split.

The computational data presented in Section~\ref{S: Data + Calculations disjoint} and the results proven in Section~\ref{S: Invariants disjoint} suggest the following: in addition to the case, where the anti-cyclotomic Iwasawa invariants are provably $0$, we should expect that $\mu(K_{\ac}/K)=\lambda(K_{\ac}/K)=0$ when $L(K)\cap K_{\ac}=K$.
The probability that the $p$-Hilbert class field is disjoint from the anti-cyclotomic tower was calculated in \eqref{prob of disjoint}.
On the other hand, the probability that we are in the situation considered in Theorem~\ref{Fujii's thm} is calculated in \eqref{prob Fujii}.
Therefore, we may propose the following:
\begin{labelledconj}{Invariants Heuristics}{H: AC}
Among the imaginary quadratic fields $K$ in which $p$ does not split,
\begin{align*}
\Prob\left(\mu(K_{\ac}/K) = \lambda(K_{\ac}/K)=0\right) &\ge\left(1-\frac{1}{p}\right) + \frac{1}{p-1}\prod_{j\geq1}\left(1-p^{-j}\right)\\
& = 1-p^{-3}-p^{-4}-p^{-5}+p^{-7}+\cdots .
\end{align*}
\end{labelledconj}

The probability is obtained by combining the case $p \nmid h_K$ and the cases where $K\subsetneq L(K)$ is either disjoint or contained in $K_{\ac}$.

\subsection{Computational Method and Data}
In this section, we provide tables which show how much of the $p$-Hilbert class field of $K$ intersects with $K_\ac$, when $p=3,5$.
This provides computational evidence in support of our heuristics, at least in the case when $p=3,5$.
Note that in each of the tables, $k$ starts with a large value because Cohen--Lenstra heuristics are known to converge very slowly.
We explain how the computations were done using PARI/GP.

The starting point is a special case of the following result of D.~Brink \cite[Theorem~2]{Bri07}.

\begin{labelledconj}{Brink's Theorem}{Thm: Brink}
Let $p$ be an odd prime, and $K$ be an imaginary quadratic field of discriminant $-d$ and class number $h=p^{\theta} u$ (such that $\gcd(p,u)=1$).
Let $\ell\neq p$ be a prime number that splits as $\mathfrak{l}\overline{\mathfrak{l}}$ in $K$, and define $\nu$ as follows:
\[
[L(K)\cap K_{\ac}: K] = p^\nu.
\]
Further, write
\[
\ell^h = \begin{cases}
a^2 + db^2 & \text{if } d\not\equiv 3 \pmod{4}\\
a^2 + ab + \left( (d+1)/4\right)b^2 & \text{if } d\equiv 3 \pmod{4}
\end{cases}
\]
with relatively prime $a,b\in \Z$.
Set\[
\omega = \begin{cases}
\sqrt{-d} & \text{if } d\not\equiv 3 \pmod{4}\\
\left(1+\sqrt{-d}\right)/2 & \text{if } d\equiv 3 \pmod{4}
\end{cases}.
\]
The following conclusions hold for $n\geq 0$.
\begin{enumerate}[\textup{(}i\textup{)}]
\item \emph{$p$ splits in $K$:} Write $(a+b\omega)^{p-1} = a^* + b^*\omega$.
Then
\[
\mathfrak{l} \text{ splits completely in } K_n \Longleftrightarrow b^* \equiv 0 \pmod{p^{n+1+\theta-\nu}}.
\]
\item \emph{$p$ is inert in $K$:} Write $(a+b\omega)^{p+1} = a^* + b^*\omega$.
Then
\[
\mathfrak{l} \text{ splits completely in } K_n \Longleftrightarrow b^* \equiv 0 \pmod{p^{n+1+\theta-\nu}}.
\]
\item \emph{$p$ ramifies in $K$:}
If the conditions of the exceptional case (case (iv) below) don't hold, then,
\[
\mathfrak{l} \text{ splits completely in } K_n \Longleftrightarrow b \equiv 0 \pmod{p^{n+\theta-\nu}}.
\]
\item \emph{$p=3$ and $d\equiv 3\pmod{9}$:} Write $(a+b\omega)^{3} = a^* + b^*\omega$.
Then
\[
\mathfrak{l} \text{ splits completely in } K_n \Longleftrightarrow b^* \equiv 0 \pmod{p^{n+2+\theta-\nu}}.
\]
\end{enumerate}
\end{labelledconj}

%When $n=0$ in the above theorem, we get a characterization of the primes $\ell\neq p$ that split in $K$.
\begin{cor}
Keep the notation introduced above and let $v_p$ denote the $p$-adic valuation.
Then,
\[
\nu = \begin{cases}
1 +\theta - \min\{v_p(b^*)\mid \ell \neq p \text{ splits in }K\} & \text{if }p \text{ splits in }K.\\
1 +\theta - \min\{v_p(b^*)\mid \ell \neq p \text{ splits in }K\} & \text{if }p \text{ is inert in }K.\\
\theta - \min\{v_p(b)\mid \ell \neq p \text{ splits in }K\} & \text{if }p \text{ ramifies in }K \text{ but not the exceptional case}.\\
2 +\theta - \min\{v_p(b^*)\mid \ell \neq p \text{ splits in }K\} & \text{if }p=3 \text{ and } d\equiv 3\pmod{9}.
\end{cases}
\]
\end{cor}
The corollary follows from the case $n=0$ of the theorem since the theorem then gives a characterization of the primes $\ell$ that split in $K$.
The degree $p^{\nu}$ of the intersection was determined using this corollary.
The fundamental discriminants $-d$ were considered in the ranges described below each table.
We found the minimum over only primes $\ell<400$ that split in $K$.
Hence, it is possible that the counts in our tables are slightly incorrect, but not enough to affect the overall agreement with our heuristic model \eqref{H1}.
The computations of the class numbers were done in PARI/GP, and therefore depend on the correctness of the Generalized Riemann Hypothesis.
We assume that this also does not affect the overall agreement with the heuristic model.

\subsubsection*{How to read the tables that follow}
The results are arranged according to the isomorphism class of $A_0$.
For example, in Table~\ref{T: 3k-1 100 to 101} in the second row, there are 31483 fields in the range mentioned such that $A_0\simeq \Z/9\Z$.
The fraction where $[L(K)\cap K_{\ac}:K]=9$ is $0.678017$, the fraction where $[L(K)\cap K_{\ac}:K]=3$ is $0.213734$, and the fraction where $L(K)$ is disjoint from $K_{\ac}$ is $0.108249$.
The row labelled `Expected' gives the fractions predicted by \eqref{H1}.

\begin{center}
\renewcommand{\arraystretch}{1.25}
\setlength{\aboverulesep}{0pt}
\setlength{\belowrulesep}{0pt}
\begin{longtable}{|c|c|c|c|c|c|c|}\caption{$d=3k-1,\quad 100,000,000 \le k \le 101,000,000$}\label{T: 3k-1 100 to 101}\\\toprule
$A_0$ &Number   &$3^4$ & $3^3$ & $3^2$ & 3 & $1$ \\ \midrule
3 &94133  & 0 & 0 & 0  & 0.674429  & 0.325571\\
Expected && 0 & 0 & 0 & 0.666667 & 0.333333 \\
\midrule
      9 & 31483&  0  & 0 &    0.678017  & 0.213734 &   0.108249 \\
Expected & & 0 & 0 & 0.666667 & 0.222222 & 0.111111\\
\midrule
      27  & 10262 &0 & 0.678718 &  0.214481 &  0.072988 &   0.033814\\
Expected & & 0 & 0.666667 & 0.222222 & 0.074074 & 0.037037\\
\midrule
  81 & 3551 & 0.669107 & 0.227542 &  0.069839   & 0.021121  &  0.012391\\
Expected & & 0.666667 & 0.222222 & 0.074074 & 0.024691 & 0.012346\\
\bottomrule
\end{longtable}
\renewcommand{\arraystretch}{1}
\end{center}

\begin{center}
\renewcommand{\arraystretch}{1.25}
\setlength{\aboverulesep}{0pt}
\setlength{\belowrulesep}{0pt}
\begin{longtable}{|c|c|c|c|c|c|c|}\caption{$d=3k+1,\quad 100,000,000 \le k \le 101,000,000$}\label{T: 3k+1 100 to 101}\\\toprule
$A_0$ &Number    & $3^4$&$ 3^3$ &  $3^2$ & $3$ & $1$ \\ \midrule
      $3$  & 93747 &  0 & 0 & 0 & 0.675328 &  0.324672  \\
Expected & & 0 & 0 & 0 &  0.666667 & 0.333333 \\
\midrule
      $9$ & 31072  & 0 & 0 & 0.677620 &  0.214663 &  0.107718 \\
Expected & & 0 & 0 & 0.666667 & 0.222222 & 0.111111\\
\midrule
     $27$ &10477  & 0 & 0.677770 &  0.216283 &  0.069486 &   0.036461  \\
Expected & & 0 & 0.666667 & 0.222222 & 0.074074 & 0.037037\\
\midrule
  81 & 3495  & 0.678970 &   0.214306 &  0.072103 &   0.023748 &   0.010873 \\
Expected & & 0.666667 & 0.222222 & 0.074074 & 0.024691 & 0.012346\\
\bottomrule
\end{longtable}
\renewcommand{\arraystretch}{1}
\end{center}

\newpage

\begin{center}
\renewcommand{\arraystretch}{1.25}
\setlength{\aboverulesep}{0pt}
\setlength{\belowrulesep}{0pt}
\begin{longtable}{|c|c|c|c|c|c|c|}\caption{$d=9k+3,\quad 1,000,000 \le k \le 1,100,000$}\label{T: 9k+3 1 to 11}\\\toprule
$A_0$ &Number    & $3^4$&$ 3^3$ &  $3^2$ & $3$ & $1$ \\ \midrule
      $3$  &  9298  & 0 & 0 & 0 & 0.687567 &  0.312433 \\
Expected & & 0 & 0 & 0 &  0.666667 & 0.333333 \\
\midrule
      $9$ &   2961 &  0 &  0 &  0.676798 &  0.209389 &   0.113813\\
Expected & & 0 & 0 & 0.666667 & 0.222222 & 0.111111\\
\midrule
     $27$ & 1103  & 0  &  0.675431 &  0.212149 &   0.074343 &    0.038078 \\
Expected & & 0 & 0.666667 & 0.222222 & 0.074074 & 0.037037\\
\midrule
  81 &   350 &  0.717143 &  0.182857 &   0.071429 &  0.020000 & 0.008571  \\
Expected & & 0.666667 & 0.222222 & 0.074074 & 0.024691 & 0.012346\\
\bottomrule
\end{longtable}
\renewcommand{\arraystretch}{1}
\end{center}

\begin{center}
\renewcommand{\arraystretch}{1.25}
\setlength{\aboverulesep}{0pt}
\setlength{\belowrulesep}{0pt}
\begin{longtable}{|c|c|c|c|c|c|c|}\caption{$d=9k+6,\quad 1,000,000 \le k \le 1,100,000$}\label{T: 9k+6 1 to 11}\\\toprule
$A_0$ &Number    & $3^4$&$ 3^3$ &  $3^2$ & $3$ & $1$ \\ \midrule
      $3$  &  9279  & 0 & 0 & 0 & 0.692532 &  0.307468 \\
Expected & & 0 & 0 & 0 &  0.666667 & 0.333333 \\
\midrule
      $9$ &  3029  & 0 & 0 & 0.689006 &  0.206339 &  0.104655 \\
Expected & & 0 & 0 & 0.666667 & 0.222222 & 0.111111\\
\midrule
     $27$ &  1065 & 0 &  0.712676 &  0.186854 &  0.066667 & 0.033803 \\
Expected & & 0 & 0.666667 & 0.222222 & 0.074074 & 0.037037\\
\midrule
  81 &   321 & 0.716511 &  0.190031 &  0.059190 &   0.018692 &  0.015576 \\
Expected & & 0.666667 & 0.222222 & 0.074074 & 0.024691 & 0.012346\\
\bottomrule
\end{longtable}
\renewcommand{\arraystretch}{1}
\end{center}

\begin{center}
\renewcommand{\arraystretch}{1.25}
\setlength{\aboverulesep}{0pt}
\setlength{\belowrulesep}{0pt}
\begin{longtable}{|c|c|c|c|c|c|c|}\caption{$d=3k-1,\quad 100,000,000 \le k \le 150,000,000$}\label{T: 3k-1 100 to 150}\\\toprule
$A_0$ &Number    & $3^4$&$ 3^3$ &  $3^2$ & $3$ & $1$ \\ \midrule
      $9\times 9$  &420 & 0 & 0 & 0.900000   & 0.078571   & 0.021429 \\
Expected && 0 & 0 & 0.888889 & 0.098765 & 0.012346 \\
\midrule
      $27\times 9$ & 215  & 0  & 0.688372   & 0.297674    &0.013953 & 0.000000\\
Expected & & 0 & 0.666667 & 0.296296  & 0.032922 & 0.004115 \\
\midrule
      $81\times 9$ & 72   & 0.555556    & 0.333333   & 0.111111 & 0.000000 & 0.000000 \\
Expected &  & 0.666667 & 0.222222 & 0.098765 & 0.010974 & 0.001372\\
\bottomrule
\end{longtable}
\renewcommand{\arraystretch}{1}
\end{center}

\begin{center}
\renewcommand{\arraystretch}{1.25}
\setlength{\aboverulesep}{0pt}
\setlength{\belowrulesep}{0pt}
\begin{longtable}{|c|c|c|c|c|}\caption{$d=3k+1,\quad 10,000,000 \le k \le 100,000,000$}\label{T: 3k+1 10 to 15}\\\toprule
$A_0$ &Number   & $3^2$ & 3 & $1$ \\\midrule
      $3\times 3\times 3$  & 597  & 0  & 0.988275  & 0.011725  \\
Expected & & 0 & 0.962963 & 0.037037 \\
\midrule
      $9\times 3\times 3$ &287   & 0.728223  & 0.268293 &  0.003484 \\
Expected &  & 0.666667 & 0.320988  & 0.012346  \\
\bottomrule
\end{longtable}
\renewcommand{\arraystretch}{1}
\end{center}

\begin{center}
\renewcommand{\arraystretch}{1.25}
\setlength{\aboverulesep}{0pt}
\setlength{\belowrulesep}{0pt}
\begin{longtable}{|c|c|c|c|c|c|c|}\caption{$d=5k+2,\quad 1,000,000 \le k \le 2,000,000$}\label{T: 5k+2 1 to 2}\\\toprule
$A_0$ &Number    & $5^4$&$ 5^3$ &  $5^2$ & $5$ & $1$ \\
\midrule
      $5$  &39002 &  0  & 0 & 0 &0.802779 &  0.197221  \\
Expected & & 0 & 0 & 0 &  0.800000 & 0.200000 \\
\midrule
      $25$ & 7639  & 0 & 0 &  0.807566 &  0.154078 &  0.038356  \\
Expected & & 0 & 0 & 0.800000 & 0.160000 & 0.040000\\
\midrule
     $125$ &1383 &0 &  0.806218 & 0.151844 &  0.033984 &  0.007954\\
Expected &  & 0 &  0.800000 & 0.160000  & 0.032000 & 0.008000\\
\midrule
 625 & 203 &   0.753695 &  0.201970 &  0.044335 & 0.000000 & 0.000000\\
Expected & & 0.800000 & 0.160000 & 0.032000 &  0.006400 & 0.001600 \\
\bottomrule
\end{longtable}
\renewcommand{\arraystretch}{1}
\end{center}

\begin{center}
\renewcommand{\arraystretch}{1.25}
\setlength{\aboverulesep}{0pt}
\setlength{\belowrulesep}{0pt}
\begin{longtable}{|c|c|c|c|c|c|}\caption{$d=5k,\quad 10,000,000 \le k \le 20,000,000$}\label{T: 5k 10 to 20}\\\toprule
$A_0$ &Number    & $ 5^3$ &  $5^2$ & $5$ & $1$ \\
\midrule
      $5\times 5$  & 2122 & 0 & 0 &  0.967484 &  0.032516  \\
Expected & & 0 & 0 & 0.960000 &  0.040000 \\
\midrule
      $25\times 5$ &493  & 0 &  0.833671 &  0.164300 &  0.002028  \\
Expected &  & 0 & 0.800000  & 0.192000 & 0.008000  \\
\midrule
     $125\times 5$ & 48  &  0.812500 &  0.145833 &  0.041667  & 0.000000\\
Expected &  & 0.800000 & 0.160000  & 0.038400 & 0.001600  \\
\bottomrule
\end{longtable}
\renewcommand{\arraystretch}{1}
\end{center}

\section*{Acknowledgements}
DK was supported by a PIMS postdoctoral fellowship at the time of writing the paper.
We thank Reinier Br{\"o}ker for helpful discussions.
We are grateful to the anonymous referees for their timely reading of the manuscript and for their valuable comments.

\bibliographystyle{amsalpha}
\bibliography{references}
\end{document}